\documentclass[reqno]{amsart}
\usepackage{amssymb,graphicx,setspace}
\usepackage{ifpdf}
\ifpdf
 \usepackage[hyperindex]{hyperref}
\else
 \expandafter\ifx\csname dvipdfm\endcsname\relax
 \usepackage[hypertex,hyperindex]{hyperref}
 \else
 \usepackage[dvipdfm,hyperindex]{hyperref}
 \fi
\fi
\allowdisplaybreaks[4]
\numberwithin{equation}{section}
\DeclareMathOperator{\td}{d\mspace{-1mu}}
\newtheorem{thm}{Theorem}[section]
\newtheorem{lem}{Lemma}[section]

\theoremstyle{definition}
\newtheorem{rem}{Remark}[section]
\newtheorem{dfn}{Definition}[section]

\begin{document}

\title[Bernstein functions and integral representations of means]
{Some Bernstein functions and integral representations concerning harmonic and geometric means}

\author[F. Qi]{Feng Qi}
\address[Qi]{School of Mathematics and Informatics, Henan Polytechnic University, Jiaozuo City, Henan Province, 454010, China}
\email{\href{mailto: F. Qi <qifeng618@gmail.com>}{qifeng618@gmail.com}, \href{mailto: F. Qi <qifeng618@hotmail.com>}{qifeng618@hotmail.com}, \href{mailto: F. Qi <qifeng618@qq.com>}{qifeng618@qq.com}}
\urladdr{\url{http://qifeng618.wordpress.com}}

\author[X.-J. Zhang]{Xiao-Jing Zhang}
\address[Zhang]{Department of Mathematics, School of Science, Tianjin Polytechnic University, Tianjin City, 300387, China}
\email{\href{mailto: X.-J. Zhang <xiao.jing.zhang@qq.com>}{xiao.jing.zhang@qq.com}}

\author[W.-H. Li]{Wen-Hui Li}
\address[Li]{Department of Mathematics, School of Science, Tianjin Polytechnic University, Tianjin City, 300387, China}
\email{\href{mailto: W.-H. Li <wen.hui.li@foxmail.com>}{wen.hui.li@foxmail.com}}

\begin{abstract}
It is general knowledge that the harmonic mean $H(x,y)=\frac2{\frac1x+\frac1y}$ and that the geometric mean $G(x,y)=\sqrt{xy}\,$, where $x$ and $y$ are two positive numbers. In the paper, the authors show by several approaches that the harmonic mean $H_{x,y}(t)=H(x+t,y+t)$ and the geometric mean $G_{x,y}(t)=G(x+t,y+t)$ are all Bernstein functions of $t\in(-\min\{x,y\},\infty)$ and establish integral representations of the means $H_{x,y}(t)$ and $G_{x,y}(t)$.
\end{abstract}



\thanks{This paper was typeset using \AmS-\LaTeX}

\maketitle

\section{Introduction}

\subsection{Some definitions}
We recall some notions and definitions.

\begin{dfn}[\cite{mpf-1993, widder}]
A function $f$ is said to be completely monotonic on an interval $I\subseteq\mathbb{R}$ if $f$ has derivatives of all orders on $I$ and
\begin{equation}
(-1)^{n}f^{(n)}(t)\ge0
\end{equation}
for all $t\in I$ and $n\in\{0\}\cup\mathbb{N}$.
\end{dfn}

\begin{dfn}[\cite{Atanassov}]
If $f^{(k)}(t)$ for some nonnegative integer $k$ is completely monotonic on an interval $I\subseteq\mathbb{R}$, but $f^{(k-1)}(t)$ is not completely monotonic on $I$, then $f(t)$ is called a completely monotonic function of $k$-th order on an interval $I$.
\end{dfn}

\begin{dfn}[\cite{compmon2, minus-one}]
A function $f$ is said to be logarithmically completely monotonic on an interval $I\subseteq\mathbb{R}$ if its logarithm $\ln f$ satisfies
\begin{equation}
(-1)^k[\ln f(t)]^{(k)}\ge0
\end{equation}
for all $t\in I$ and $k\in\mathbb{N}$.
\end{dfn}

\begin{dfn}[\cite{Schilling-Song-Vondracek-2010, widder}]
A function $f:I\subseteq(-\infty,\infty)\to[0,\infty)$ is called a Bernstein function on $I$ if $f(t)$ has derivatives of all orders and $f'(t)$ is completely monotonic on $I$.
\end{dfn}

\begin{dfn}[{\cite{Schilling-Song-Vondracek-2010}}]\label{Stieltjes-dfn}
A Stieltjes function is a function $f:(0,\infty)\to[0,\infty)$ which can be written in the form
\begin{equation}\label{dfn-stieltjes}
f(x)=\frac{a}x+b+\int_0^\infty\frac1{s+x}{\td\mu(s)},
\end{equation}
where $a,b$ are nonnegative constants and $\mu$ is a nonnegative measure on $(0,\infty)$ such that $\int_0^\infty\frac1{1+s}\td\mu(s)<\infty$.
\end{dfn}

\begin{dfn}[\cite{psi-proper-fraction-degree-two.tex}]\label{x-degree-dfn}
Let $f(x)$ be a nonnegative function and have derivatives of all orders on $(0,\infty)$. A number $r\in\mathbb{R}\cup\{\pm\infty\}$ is said to be the completely monotonic degree of $f(x)$ with respect to $x\in(0,\infty)$ if $x^rf(x)$ is a completely monotonic function on $(0,\infty)$ but $x^{r+\varepsilon}f(x)$ is not for any positive number $\varepsilon>0$.
\end{dfn}

In what follows, for convenience, we denote the sets of completely monotonic functions on $I\subseteq\mathbb{R}$, logarithmically completely monotonic functions on $I\subseteq\mathbb{R}$, Stieltjes functions, and Bernstein functions on $I\subseteq\mathbb{R}$ by $\mathcal{C}[I]$, $\mathcal{L}[I]$, $\mathcal{S}$, and $\mathcal{B}[I]$ respectively.

\subsection{Some relationships and a characterization}
Now we briefly describe some basic relationships between the above defined classes of functions and list a characterization of Bernstein functions on $(0,\infty)$.
\par
In~\cite{CBerg, absolute-mon-simp.tex, compmon2, minus-one}, any logarithmically completely monotonic function on an interval $I$ was once again proved to be completely monotonic on $I$. In~\cite{CBerg}, the set of all Stieltjes functions was proved to be a subset of all logarithmically completely monotonic functions on $(0,\infty)$. See also~\cite[Remark~4.8]{Open-TJM-2003-Banach.tex}. Conclusively,
\begin{equation}\label{S-L-C-relation}
\mathcal{S}\subset\mathcal{L}[(0,\infty)]\subset\mathcal{C}[(0,\infty)].
\end{equation}
\par
It is obvious that any nonnegative completely monotonic function of first order is a Bernstein function.
\par
The relation between Bernstein functions and logarithmically completely monotonic functions was discovered in~\cite[pp.~161\nobreakdash--162, Theorem~3]{Chen-Qi-Srivastava-09.tex} and~\cite[p.~45, Proposition~5.17]{Schilling-Song-Vondracek-2010}, which reads that the reciprocal of any positive Bernstein function is logarithmically completely monotonic. In other words,
\begin{equation}
0<f\in\mathcal{B}[I]\Longrightarrow \frac1f\in\mathcal{L}[I].
\end{equation}
\par
A relation between $\mathcal{S}$ and $\mathcal{B}[(0,\infty)]$ was given by~\cite[Theorem~5.4]{Berg-SPBS-08} which may be recited as
\begin{equation}\label{SB-relation}
0<f\in\mathcal{S}\Longrightarrow \frac1f\in\mathcal{B}[(0,\infty)].
\end{equation}
\par
It is easy to see that the degree of any completely monotonic function on $(0,\infty)$ is at least zero. Conversely, if a nonnegative function $f(x)$ on $(0,\infty)$ has a nonnegative degree $r$, then it must be a completely monotonic function on $(0,\infty)$. See~\cite[p.~9890]{psi-proper-fraction-degree-two.tex}.
\par
Bernstein functions can be characterized by~\cite[p.~15, Theorem~3.2]{Schilling-Song-Vondracek-2010} which states that a function $f:(0,\infty)\to\mathbb{R}$ is a Bernstein function if and only if it admits the representation
\begin{equation}\label{Bernstein-Characterize}
f(x)=a+bx+\int_0^\infty\bigl(1-e^{-xt}\bigr)\td\mu(t),
\end{equation}
where $a,b\ge0$ and $\mu$ is a measure on $(0,\infty)$ satisfying
$
\int_0^\infty\min\{1,t\}\td\mu(t)<\infty.
$
\par
For information on characterizations of the classes $\mathcal{C}[(0,\infty)]$ and $\mathcal{L}[(0,\infty)]$, please refer to related texts in~\cite{CBerg, Schilling-Song-Vondracek-2010, widder} and references cited therein.

\subsection{Some means}

We recall from~~\cite{stolarsky-mean} that the extended mean value $E(r,s;x,y)$ may be defined by
\begin{align}
E(r,s;x,y)&=\biggl[\frac{r(y^s-x^s)} {s(y^r-x^r)}\biggr]^{{1/(s-r)}}, & rs(r-s)(x-y)&\ne 0; \\
E(r,0;x,y)&=\biggl[\frac{y^r-x^r}{r(\ln y-\ln x)}\biggr]^{{1/r}}, & r(x-y)&\ne 0; \\
E(r,r;x,y)&=\frac1{e^{1/r}}\biggl(\frac{x^{x^r}}{y^{y^r}}\biggr)^{ {1/(x^r-y^r)}},& r(x-y)&\ne 0; \\
E(0,0;x,y)&=\sqrt{xy}\,, & x&\ne y; \\
E(r,s;x,x)&=x, & x&=y;\notag
\end{align}
where $x,y$ are positive numbers and $r,s\in\mathbb{R}$. Because this mean was first defined in~\cite{stolarsky-mean}, so it is also called Stolarsky's mean by a number of mathematicians. Many special means with two positive variables are special cases of $E$, for example,
\begin{align*}
E(r,2r;x,y)&=M_r(x,y), &&(\text{power mean}) \\
E(1,p;x,y)&=L_p(x,y), &&(\text{generalized logarithmic mean}) \\
E(1,1;x,y)&=I(x,y), &&(\text{exponential mean}) \\
E(1,2;x,y)&=A(x,y), &&(\text{arithmetic mean}) \\
E(0,0;x,y)&=G(x,y), &&(\text{geometric mean}) \\
E(-2,-1;x,y)&=H(x,y), &&(\text{harmonic mean}) \\
E(0,1;x,y)&=L(x,y). &&(\text{logarithmic mean})
\end{align*}
For more information on $E$, please refer to the monograph~\cite{bullenmean}, the papers~\cite{emv-log-convex-simple.tex, Guo-Qi-Filomat-2011-May-12.tex, cubo, mon-element-exp-gen.tex}, and a lot of closely-related references therein.

\subsection{The arithmetic mean is a Bernstein function}

It is easy to see that the arithmetic mean
\begin{equation*}
A_{x,y}(t)=A(x+t,y+t)=A(x,y)+t
\end{equation*}
is a trivial Bernstein function of $t\in(-\min\{x,y\},\infty)$ for $x,y>0$.

\subsection{The exponential mean is a Bernstein function}

In~\cite[p.~116, Remark~6]{new-upper-kershaw-JCAM.tex}, it was pointed out that,
\begin{enumerate}
\item
by standard arguments, it is easy to verify that the reciprocal of the exponential mean
\begin{equation}
I_{x,y}(t)=I(x+t,y+t)=\frac1e\biggl[\frac{(x+t)^{x+t}}{(y+t)^{y+t}}\biggr]^{ {1/(x-y)}}
\end{equation}
for $x,y>0$ with $x\ne y$ is a logarithmically completely monotonic function of $t\in(-\min\{x,y\},\infty)$;
\item
from the newly-discovered integral representation
\begin{equation}
I(x,y)=\exp\biggl(\frac1{y-x}\int_x^y\ln u\td u\biggr),
\end{equation}
it is easy to obtain that the exponential mean $I_{x,y}(t)$ for $t>-\min\{x,y\}$ with $x\ne y$ is also a completely monotonic function of first order (that is, a Bernstein function).
\end{enumerate}

\subsection{The logarithmic mean is a Bernstein function}

In~\cite[p.~616, Remark~3.7]{gamma-psi-batir.tex}, the logarithmic mean
\begin{equation}
L_{x,y}(t)=L(x+t,y+t)
\end{equation}
was proved to be increasing and concave in $t>-\min\{x,y\}$ for $x,y>0$ with $x\ne y$.
\par
More strongly, the logarithmic mean $L_{x,y}(t)$ was proved in~\cite[Theorem~1]{log-mean-comp-mon.tex-mia} to be a completely monotonic function of first order on $(-\min\{x,y\},\infty)$ for $x,y>0$ with $x\ne y$. Therefore, the logarithmic mean $L_{x,y}(t)$ is a Bernstein function of $t\in(-\min\{x,y\},\infty)$.

\begin{rem}
By~\cite[pp.~161\nobreakdash--162, Theorem~3]{Chen-Qi-Srivastava-09.tex} or~\cite[p.~45, Proposition~5.17]{Schilling-Song-Vondracek-2010}, the logarithmically complete monotonicity of the exponential mean $I_{x,y}(t)$ and the logarithmic mean $L_{x,y}(t)$ can be deduced respectively from their common property that they are Bernstein functions.
\end{rem}

\subsection{Main results}

The goals of this paper are to prove that the harmonic mean
\begin{equation}\label{harmnic=x-y+t-eq}
H_{x,y}(t)=H(x+t,y+t)=\frac2{\frac1{x+t}+\frac1{y+t}}
\end{equation}
and the geometric mean
\begin{equation}\label{G(x=y)(t)}
G_{x,y}(t)=G(x+t,y+t)=\sqrt{(x+t)(y+t)}\,
\end{equation}
are all Bernstein functions of $t$ on $(-\min\{x,y\},\infty)$ for $x,y>0$ with $x\ne y$, and to establish integral representations of $H_{x,y}(t)$ and $G_{x,y}(t)$.

\section{Lemmas}

In order to prove our main results, the following lemmas are needed.

\begin{lem}\label{Geom-Mean-CM-lem-h(t)}
For $i\in\mathbb{N}$, the $i$-th derivatives of the functions
\begin{equation}
h(t)=\sqrt{1+\frac1t}\,,
\end{equation}
the reciprocal $\frac1{h(t)}$, and
\begin{equation}
H(t)=h(t)+\frac1{h(t)}
\end{equation}
on $(0,\infty)$ may be computed by
\begin{align}\label{h(t)-i-D}
h^{(i)}(t)&=\frac{(-1)^i}{2^it^{i+1}(1+t)^{i-1}h(t)}\sum_{k=0}^{i-1}a_{i,k}t^k,\\
\biggl[\frac1{h(t)}\biggr]^{(i)}&=\frac{(-1)^{i+1}}{2^it^i(1+t)^ih(t)}\sum_{k=0}^{i-1}b_{i,k}t^k, \label{h(t)-recip-i-D}\\
H^{(i)}(t)&=\frac{(-1)^i}{2^it^{i+1}(1+t)^{i}h(t)} \sum_{k=0}^{i-1}c_{i,k}t^k, \label{H(t)-expression}
\end{align}
where
\begin{align}\label{a=(i-k)-eq}
a_{i,k}&=\frac{(i-1)!i!(2i-2k-1)!!}{(i-k-1)!(i-k)!k!}2^{k},\\
b_{i,k}&=\frac{(i-1)!i!(2i-2k-3)!!}{(i-k-1)!(i-k)!k!}2^k, \label{b=(i-k)-eq}\\
c_{i,k}&=\frac{(i-1)!(i+1)!(2i-2k-1)!!}{(i-k-1)!(i-k+1)!k!}2^{k}.\label{c=(i-k)-eq}
\end{align}
Consequently, the functions $h(t)$ and $H(t)$ are completely monotonic on $(0,\infty)$, and the reciprocal $\frac1{h(t)}$ is a Bernstein function on $(0,\infty)$.
\end{lem}

\begin{proof}[Inductive proof of Lemma~\ref{Geom-Mean-CM-lem-h(t)}]
A direct calculation yields $h'(t)=-\frac1{2t^2h(t)}$, which means that
\begin{equation}\label{a(1,0)}
a_{1,0}=1.
\end{equation}
So, the formulas~\eqref{h(t)-i-D} and~\eqref{a=(i-k)-eq} are valid for $i=1$ and $k=0$.
\par
Differentiating on both sides of~\eqref{h(t)-i-D} gives
\begin{align*}
h^{(i+1)}(t)=\bigl[h^{(i)}(t)\bigr]'
&=\Biggl[\frac{(-1)^i}{2^it^{i+1}(1+t)^{i-1}h(t)}\sum_{k=0}^{i-1}a_{i,k}t^k\Biggr]'\\
&=\frac{(-1)^{i+1}}{2^{i+1}t^{i+2}(1+t)^{i}h(t)}\sum_{k=0}^{i-1}[1+2(i-k)+2(2i-k)t]a_{i,k}t^k\\
&=\frac{(-1)^{i+1}}{2^{i+1}t^{i+2}(1+t)^{i}h(t)}\sum_{k=0}^{i}a_{i+1,k}t^k.
\end{align*}
Because
\begin{gather*}
\sum_{k=0}^{i-1}[1+2(i-k)+2(2i-k)t]a_{i,k}t^k
=\sum_{k=0}^{i-1}[1+2(i-k)]a_{i,k}t^k +\sum_{k=0}^{i-1}2(2i-k)a_{i,k}t^{k+1}\\
=\sum_{k=0}^{i-1}[1+2(i-k)]a_{i,k}t^k +\sum_{k=1}^{i}2(2i-k+1)a_{i,k-1}t^{k}\\
=(1+2i)a_{i,0}+\sum_{k=1}^{i-1}\{[1+2(i-k)]a_{i,k}+2(2i-k+1)a_{i,k-1}\}t^k +2(i+1)a_{i,i-1}t^{i},
\end{gather*}
we obtain
\begin{align}
a_{i+1,0}&=(1+2i)a_{i,0},\label{a(i+1,0)}\\
a_{i+1,i}&=2ia_{i,i-1},\label{a(i+1,i)}
\end{align}
and, for $0<k<i$,
\begin{equation}\label{a(i+1,k)}
a_{i+1,k}=[1+2(i-k)]a_{i,k}+2(2i-k+1)a_{i,k-1}.
\end{equation}
\par
Combining~\eqref{a(1,0)} with~\eqref{a(i+1,0)} and~\eqref{a(i+1,i)} results in
\begin{equation}\label{a=i-0}
a_{i,0}=(2i-1)!!
\end{equation}
and
\begin{equation}\label{a=i=i-1}
a_{i,i-1}=2^{i-1}i!.
\end{equation}
\par
Taking $k=i-1$ in~\eqref{a(i+1,k)} and using~\eqref{a=i=i-1} give
\begin{equation}\label{a=i+1=i-1}
a_{i+1,i-1}=3a_{i,i-1}+2(i+2)a_{i,i-2}=3\cdot2^{i-1}i!+2(i+2)a_{i,i-2}.
\end{equation}
From~\eqref{a=i-0}, it is easily deduced that $a_{2,0}=3$. Substituting this into~\eqref{a=i+1=i-1} and recurring repeatedly lead to
\begin{equation}\label{a=i=i-2}
a_{i,i-2}=3(i-1)2^{i-3}i!.
\end{equation}
\par
Taking $k=i-2$ in~\eqref{a(i+1,k)} and using~\eqref{a=i=i-2} show
\begin{equation}\label{a=i+1=i-2}
a_{i+1,i-2}=5a_{i,i-2}+2(i+3)a_{i,i-3}=15(i-1)2^{i-3}i!+2(i+3)a_{i,i-3}.
\end{equation}
From~\eqref{a=i-0}, it is readily deduced that $a_{3,0}=15$. Substituting this into~\eqref{a=i+1=i-2} and recurring repeatedly reveal
\begin{equation}\label{a=i=i-3}
a_{i,i-3}=5(i-2)(i-1)2^{i-5}i!.
\end{equation}
\par
Taking $k=i-3$ in~\eqref{a(i+1,k)} and using~\eqref{a=i=i-3} show
\begin{equation}\label{a=i+1=i-3}
a_{i+1,i-3}=7a_{i,i-3}+2(i+4)a_{i,i-4}=35(i-2)(i-1)2^{i-5}i!+2(i+4)a_{i,i-4}.
\end{equation}
From~\eqref{a=i-0}, it is immediately obtained that $a_{4,0}=105$. Substituting this into~\eqref{a=i+1=i-3} and recurring repeatedly yield
\begin{equation}\label{a=i=i-4}
a_{i,i-4}=\frac{35}{3} (i-3) (i-2) (i-1) 2^{i-8} i!.
\end{equation}
\par
By the same arguments as above, we may obtain
\begin{equation}\label{a=i=i-5}
a_{i,i-5}=21(i-4)(i-3)(i-2)(i-1)2^{i-11}i!
\end{equation}
and
\begin{equation}\label{a=i=i-6}
a_{i,i-6}=\frac{77}5(i-5)(i-4)(i-3)(i-2)(i-1)2^{i-13}i!.
\end{equation}
\par
Inductively, we can derive that
\begin{equation}\label{a=i=i-k}
a_{i,i-k}=\lambda_{i,i-k}\frac{(i-1)!}{(i-k)!}2^{i-k}i!
\end{equation}
for $0<k<i$. Specially, we have
\begin{equation}
\begin{aligned}\label{b=i-k=eq}
\lambda_{i,i-1}&=1, & \lambda_{i,i-2}&=\frac32, & \lambda_{i,i-3}&=\frac54, \\
\lambda_{i,i-4}&=\frac{35}{3\cdot2^4}, & \lambda_{i,i-5}&=\frac{21}{2^6}, & \lambda_{i,i-6}&=\frac{77}{5\cdot2^7}.
\end{aligned}
\end{equation}
\par
Replacing $k$ by $i-\ell$ in~\eqref{a=i=i-k} yields
\begin{equation}\label{a=i=ell}
a_{i,\ell}=\lambda_{i,\ell}\frac{(i-1)!}{\ell!}2^\ell i!
\end{equation}
for $0<\ell<i$. Substituting~\eqref{a=i=ell} into~\eqref{a(i+1,k)} leads to
\begin{equation}\label{b(i=ell)}
[1+2(i-\ell)]\lambda_{i,\ell}+\ell(2i-\ell+1)\lambda_{i,\ell-1}=i(i+1)\lambda_{i+1,\ell}
\end{equation}
for $0<\ell<i$. The equality~\eqref{b(i=ell)} is equivalent to
\begin{equation}\label{b=(i-k)-recur}
(1+2k)\lambda_{i,i-k}+(i-k)(i+k+1)\lambda_{i,i-k-1}=i(i+1)\lambda_{i+1,i-k}
\end{equation}
for $0<k<i$.
\par
The quantities in~\eqref{b=i-k=eq} implies that $\lambda_{i,i-k}=\mu_k$, that is, $\lambda_{i,i-k}$ is independent of $i$. Then the equality~\eqref{b=(i-k)-recur} may be written as
\begin{equation}\label{b=(i-k)-recur-without=i}
(1+2k)\mu_k=[i(i+1)-(i-k)(i+k+1)]\mu_{k+1}=k(1+k)\mu_{k+1}
\end{equation}
for $0<k<i$. Recurring~\eqref{b=(i-k)-recur-without=i} by $\mu_1=\lambda_{i,i-1}=1$ reveals
\begin{equation}\label{lambda-express}
\mu_k=\lambda_{i,i-k}=\frac{(2k-1)!!}{(k-1)!k!}
\end{equation}
for $0<k<i$. As a result, by~\eqref{lambda-express}, we conclude that
\begin{equation}\label{a=i=i-k-lambda}
a_{i,i-k}=\frac{(2k-1)!!}{(k-1)!k!}\frac{(i-1)!}{(i-k)!}2^{i-k}i!
\end{equation}
for $0<k<i$. Replacing $k$ by $i-\ell$ in~\eqref{a=i=i-k-lambda} shows
\begin{equation}\label{a=i=i-k-ell}
a_{i,\ell}=\frac{(2i-2\ell-1)!!}{(i-\ell-1)!(i-\ell)!}\frac{(i-1)!}{\ell!}2^{\ell}i!
\end{equation}
for $0<\ell<i$. It is easy to verify that the sequence~\eqref{a=i=i-k-ell} for $0\le\ell\le i-1$ meets the recursion formulas~\eqref{a(i+1,0)}, \eqref{a(i+1,i)}, and~\eqref{a(i+1,k)}. The formulas~\eqref{h(t)-i-D} and~\eqref{a=(i-k)-eq} for general terms are thus proved.
\par
It is obvious that
$
h'(t)=-\frac1{2t^2h(t)}
$
which is equivalent to
$
\frac1{h(t)}=-2t^2h'(t).
$
Therefore, using the formulas~\eqref{h(t)-i-D} and~\eqref{a=(i-k)-eq} just verified, we have
\begin{align*}
\biggl[\frac1{h(t)}\biggr]^{(i)}&=-2\bigl[t^2h'(t)\bigr]^{(i)}\\
&=-2\sum_{\ell=0}^i\binom{i}{\ell}\bigl(t^2\bigr)^{(\ell)}h^{(i-\ell+1)}(t)\\
&=-2\biggl[\binom{i}{0}t^2h^{(i+1)}(t)+2\binom{i}{1}th^{(i)}(t)+2\binom{i}{2}h^{(i-1)}(t)\biggr]\\
&=-2\Biggl[\frac{(-1)^{i+1}}{2^{i+1}t^{i}(1+t)^{i}h(t)}\sum_{k=0}^{i}a_{i+1,k}t^k +\frac{(-1)^ii}{2^{i-1}t^{i}(1+t)^{i-1}h(t)}\sum_{k=0}^{i-1}a_{i,k}t^k \\
&\quad+\frac{(-1)^{i-1}(i-1)i}{2^{i-1}t^{i}(1+t)^{i-2}h(t)}\sum_{k=0}^{i-2}a_{i-1,k}t^k\Biggr]\\
&=\frac{(-1)^{i+1}}{2^it^i(1+t)^ih(t)}\Biggl[4i(1+t) \sum_{k=0}^{i-1}a_{i,k}t^k-\sum_{k=0}^{i}a_{i+1,k}t^k\\
&\quad-4(i-1)i(1+t)^2\sum_{k=0}^{i-2}a_{i-1,k}t^k\Biggr]\\
&=\frac{(-1)^{i+1}}{2^it^i(1+t)^ih(t)}\Biggl\{4ia_{i,0}-4i(i-1)a_{i-1,0}-a_{i+1,0}\\
&\quad+[4i(a_{i,1}+a_{i,0})-4i(i-1)(a_{i-1,1}+2a_{i-1,0})-a_{i+1,1}]t \\
&\quad+[4i(a_{i,i-1}+a_{i,i-2})-4i(i-1)(a_{i-1,i-3}+2a_{i-1,i-2})-a_{i+1,i-1}]t^{i-1}\\
&\quad+[4ia_{i,i-1}-4i(i-1)a_{i-1,i-2}-a_{i+1,i}]t^i +\sum_{k=2}^{i-2}[4i(a_{i,k}+a_{i,k-1}) \\
&\quad-4i(i-1)(a_{i-1,k}+2a_{i-1,k-1}+a_{i-1,k-2})-a_{i+1,k}]t^k\Biggr\}\\
&=\frac{(-1)^{i+1}}{2^it^i(1+t)^ih(t)}\sum_{k=0}^{i-1}\frac{(i-1)!i!(2i-2k-3)!!}{(i-k-1)!(i-k)!k!}2^kt^k.
\end{align*}
Hence, the general formulas~\eqref{h(t)-recip-i-D} and~\eqref{b=(i-k)-eq} are obtained.
\par
Adding the two formulas~\eqref{h(t)-i-D} and~\eqref{h(t)-recip-i-D} yields
\begin{gather*}
h^{(i)}(t)+\biggl[\frac1{h(t)}\biggr]^{(i)}=\frac{(-1)^i}{2^it^{i+1}(1+t)^{i}h(t)} \Biggl[(1+t)\sum_{k=0}^{i-1}a_{i,k}t^k -t\sum_{k=0}^{i-1}b_{i,k}t^k\Biggr]\\
\begin{aligned}
&=\frac{(-1)^i}{2^it^{i+1}(1+t)^{i}h(t)} \Biggl[\sum_{k=0}^{i-1}(a_{i,k}-b_{i,k})t^{k+1}+\sum_{k=0}^{i-1}a_{i,k}t^{k}\Biggr]\\
&=\frac{(-1)^i}{2^it^{i+1}(1+t)^{i}h(t)} \Biggl[\sum_{k=1}^{i}(a_{i,k-1}-b_{i,k-1})t^k+\sum_{k=0}^{i-1}a_{i,k}t^{k}\Biggr]
\end{aligned}\\
\begin{aligned}
&=\frac{(-1)^i}{2^it^{i+1}(1+t)^{i}h(t)} \Biggl[a_{i,0}+\sum_{k=1}^{i-1}(a_{i,k-1}-b_{i,k-1}+a_{i,k})t^k+(a_{i,i-1}-b_{i,i-1})t^i\Biggr]\\
&=\frac{(-1)^i}{2^it^{i+1}(1+t)^{i}h(t)} \Biggl\{(2i-1)!! +\sum_{k=1}^{i-1}\frac{(i-1)!(i+1)!(2i-2k-1)!!}{(i-k-1)!(i-k+1)!k!}2^{k}t^k\Biggr\}\\
&=\frac{(-1)^i}{2^it^{i+1}(1+t)^{i}h(t)} \sum_{k=0}^{i-1}\frac{(i-1)!(i+1)!(2i-2k-1)!!}{(i-k-1)!(i-k+1)!k!}2^{k}t^k.
\end{aligned}
\end{gather*}
This implies that the function $H(t)$ is completely monotonic on $(0,\infty)$. The proof of Lemma~\ref{Geom-Mean-CM-lem-h(t)} is completed.
\end{proof}

\begin{proof}[Short proofs of a part of Lemma~\ref{Geom-Mean-CM-lem-h(t)}]
In~\cite[p.~13, Remark~2.4]{Schilling-Song-Vondracek-2010}, it was collected as an example that the function $\frac1{a+t}$ is a Stieltjes function for $a>0$. The property~(iv) in Section~3 of~\cite{Berg-SPBS-08} (See also the property~(vii) in~\cite[Theorem~1.3]{Kalugin-Jeffrey-Corless-Borwein}) reads that if $f\in\mathcal{S}$ then $f^\alpha\in\mathcal{S}$ for $0\le \alpha\le1$. Specially for $a=1$ and $\alpha=\frac12$, we have $h_1(t)=\frac1{\sqrt{1+t}\,}\in\mathcal{S}$. The property~(i) in Section~3 of~\cite{Berg-SPBS-08} (See also the property~(i) in~\cite[Theorem~1.3]{Kalugin-Jeffrey-Corless-Borwein}) states that if $f\in\mathcal{S}\setminus\{0\}$ then $\frac1{f(1/x)}\in\mathcal{S}$. Applying this property to $h_1(t)$ brings out
\begin{equation}\label{h(t)inS-eq}
h(t)=\frac1{h_1(1/t)}\in\mathcal{S}
\end{equation}
which means, by the relation from the very ends of the inclusions~\eqref{S-L-C-relation}, that $h(t)\in\mathcal{C}[(0,\infty)]$ and, by the relation~\eqref{SB-relation}, that $\frac1{h(t)}\in\mathcal{B}[(0,\infty)]$.
\par
In~\cite[p.~24, Remark~3.11]{Schilling-Song-Vondracek-2010}, it was listed as examples that $h_2(t)=t^\beta\in\mathcal{B}[(0,\infty)]$ for $0<\beta<1$ and $h_3(t)=\frac{t}{1+t}\in\mathcal{B}[(0,\infty)]$. The item~(iii) of Corollary~3.7 in~\cite[p.~20]{Schilling-Song-Vondracek-2010} write that if $f_1,f_2\in\mathcal{B}[(0,\infty)]$ then $f_1\circ f_2\in\mathcal{B}[(0,\infty)]$. Applying $f_1$ and $f_2$ respectively to $h_2$ and $h_3$ reveals once again that $\frac1{h(t)}=\sqrt{\frac{t}{1+t}}\,\in\mathcal{B}[(0,\infty)]$.
\par
Taking $h_3(x)=x+\frac1x$ and $h_4(t)=\frac1{h(t)}=\frac1{\sqrt{1+1/t}\,}$. It is easy to see that $h_3\in\mathcal{C}[(0,1)]$ and $0<h_4(t)<1$. A part of Theorem~3.6 in~\cite[p.~19]{Schilling-Song-Vondracek-2010} asserts that if $0<f\in\mathcal{B}[(0,\infty)]$ then $g\circ f\in\mathcal{C}[(0,\infty)]$ for every $g\in\mathcal{C}[(0,\infty)]$. Since $h_4\in\mathcal{B}[(0,\infty)]$, applying $f$ and $g$ in this assertion respectively to $h_4$ and $h_3$ leads to $H(t)=h(t)+\frac1{h(t)}\in\mathcal{C}[(0,\infty)]$. The proof of Lemma~\ref{Geom-Mean-CM-lem-h(t)} is completed.
\end{proof}

\begin{lem}\label{lem2-zhang-li-qi-eq}
For $z\in\mathbb{C}\setminus(-\infty,0]$, the complex functions $h(z)$ and $\frac1{h(z)}$ have integral representations
\begin{equation}\label{thm-zhang-li-qi-eq}
h(z)=1+\frac1\pi\int_0^1\sqrt{\frac1u-1}\,\frac{\td u}{u+z}
\end{equation}
and
\begin{equation}\label{thm-zhang-li-qi-eq-frac}
\frac1{h(z)}=1-\frac1\pi\int_0^1\frac1{\sqrt{\frac1u-1}\,}\frac{\td u}{u+z}.
\end{equation}
Consequently, the functions $h(t)$ and $1-\frac1{h(t)}$ are Stieltjes functions and the complex function $H(z)$ has the integral integral representation
\begin{equation}\label{H(z)=int-expression}
H(z)=2+\frac1\pi\int_0^\infty\rho(s)e^{-zs}\td s
\end{equation}
for $z\in\mathbb{C}\setminus(-\infty,0]$, where
\begin{equation}\label{rho-funct-dfn}
\rho(s)=\int_0^{1/2} q(u)\bigl[1-e^{-(1-2u)s}\bigr]e^{-us}\td u
=\int_0^{1/2} q\biggl(\frac12-u\biggr)\bigl(e^{us}-e^{-us}\bigr)e^{-s/2}\td u
\end{equation}
is nonnegative on $(0,\infty)$ and
\begin{equation}
q(u)=\sqrt{\frac1u-1}\,-\frac1{\sqrt{1/u-1}\,}
\end{equation}
on $(0,1)$.
\end{lem}

\begin{proof}[Proof by Cauchy integral formula]
By standard arguments, we immediately obtain that
\begin{gather}\label{zh(z)=0}
\lim_{z\to0}[zh(z)]=\lim_{z\to0}\sqrt{z^2+z}\,=\sqrt{\lim_{z\to0}(z^2+z)}\,=0,\\
\lim_{z\to0}\frac{z}{h(z)}=\lim_{z\to0}\sqrt{\frac{z^3}{1+z}}\, =\sqrt{\lim_{z\to0}\frac{z^3}{1+z}}\,=0, \label{frac-z=h(z)=0}\\
\lim_{z\to\infty}\sqrt{1+\frac1z}\,=\sqrt{1+\lim_{z\to\infty}\frac1z}\,=1, \label{sqrt-1=frac-z}\\
\lim_{z\to\infty}\frac1{\sqrt{1+\frac1z}}\,=\frac1{\sqrt{1+\lim_{z\to\infty}\frac1z}}\,=1, \label{sqrt-1=frac-z-recip}\\
h(\overline{z})=\overline{h(z)}, \label{h(z)-h(z)-conjugate}\\
\frac1{h(\overline{z})}=\overline{\biggl[\frac1{h(z)}\biggr]}. \label{h(z)-frac-conjugate}
\end{gather}
\par
For $t\in(0,\infty)$ and $\varepsilon>0$, we have
\begin{gather*}
h(-t+i\varepsilon)=\sqrt{1+\frac1{-t+i\varepsilon}}\,
=\sqrt{1+\frac{-t-i\varepsilon}{t^2+\varepsilon^2}}\,
=\exp\biggl[\frac12\ln\biggl(1+\frac{-t-i\varepsilon}{t^2+\varepsilon^2}\biggr)\biggr]\\
=\exp\biggl\{\frac12\biggl[\ln\biggl|\frac{t^2+\varepsilon^2-t}{t^2+\varepsilon^2} -i\frac{\varepsilon}{t^2+\varepsilon^2}\biggr| +i\arg\biggl(\frac{t^2+\varepsilon^2-t}{t^2+\varepsilon^2}-i\frac{\varepsilon} {t^2+\varepsilon^2}\biggr)\biggr]\biggr\}\\
=\exp\biggl\{\frac12\biggl[\ln p(t,\varepsilon) +i\arg\biggl(\frac{t^2+\varepsilon^2-t}{t^2+\varepsilon^2} -i\frac{\varepsilon}{t^2+\varepsilon^2}\biggr)\biggr]\biggr\}\\
=
\begin{cases}
\exp\biggl\{\dfrac12\biggl[\ln p(t,\varepsilon)\, +i\arctan\dfrac{\varepsilon}{t^2+\varepsilon^2}\biggr]\biggr\}, &t^2+\varepsilon^2-t>0,\\
\exp\biggl\{\dfrac12\biggl[\ln p(t,\varepsilon)\, +i\biggl(\arctan\dfrac{\varepsilon}{t^2+\varepsilon^2}-\pi\biggr)\biggr]\biggr\}, &t^2+\varepsilon^2-t<0,\\
\exp\biggl\{\dfrac12\biggl(\ln\dfrac{\varepsilon}{t^2+\varepsilon^2}\, -i\dfrac\pi2\biggr)\biggr\}, &t^2+\varepsilon^2-t=0,
\end{cases}
\end{gather*}
where
\begin{equation*}
p(t,\varepsilon)=\sqrt{\biggl(\dfrac{t^2+\varepsilon^2-t}{t^2+\varepsilon^2}\biggr)^2 +\biggl(\dfrac{\varepsilon}{t^2+\varepsilon^2}\biggr)^2}\,.
\end{equation*}
Hence,
\begin{equation*}
\Im h(-t+i\varepsilon)=\begin{cases}
\exp\biggl[\dfrac12\ln p(t,\varepsilon)\,\biggr] \sin\biggl(\dfrac12\arctan\dfrac{\varepsilon}{t^2+\varepsilon^2}\biggr), &t^2+\varepsilon^2-t>0;\\
\exp\biggl[\dfrac12\ln p(t,\varepsilon)\,\biggr] \sin\biggl(\dfrac12\arctan\dfrac{\varepsilon}{t^2+\varepsilon^2}-\dfrac\pi2\biggr), &t^2+\varepsilon^2-t<0;\\
-\exp\biggl(\dfrac12\ln\dfrac{\varepsilon}{t^2+\varepsilon^2}\,\biggr)\sin\dfrac\pi4, &t^2+\varepsilon^2-t=0.
\end{cases}
\end{equation*}
Accordingly,
\begin{equation}\label{im-limit}
\lim_{\varepsilon\to0^+}\Im h(-t+i\varepsilon)=
\begin{cases}
-\sqrt{\dfrac1t-1}, & 0<t<1;\\
\infty, & t=1;\\
0, & t>1.
\end{cases}
\end{equation}
Similarly, for $t\in(0,\infty)$ and $\varepsilon>0$, we have
\begin{align*}
\frac1{h(-t+i\varepsilon)}&=\exp\biggl[-\frac12\ln\biggl(1+\frac{-t-i\varepsilon} {t^2+\varepsilon^2}\biggr)\biggr]\\
&=
\begin{cases}
\exp\biggl\{-\dfrac12\biggl[\ln p(t,\varepsilon)\, +i\arctan\dfrac{\varepsilon}{t^2+\varepsilon^2}\biggr]\biggr\}, &t^2+\varepsilon^2-t>0;\\
\exp\biggl\{-\dfrac12\biggl[\ln p(t,\varepsilon)\, +i\biggl(\arctan\dfrac{\varepsilon}{t^2+\varepsilon^2}-\pi\biggr)\biggr]\biggr\}, &t^2+\varepsilon^2-t<0;\\
\exp\biggl\{-\dfrac12\biggl(\ln\dfrac{\varepsilon}{t^2+\varepsilon^2}\, -i\dfrac\pi2\biggr)\biggr\}, &t^2+\varepsilon^2-t=0.
\end{cases}
\end{align*}
Therefore,
\begin{multline*}
\Im\biggl[\frac1{h(-t+i\varepsilon)}\biggr]=\\
\begin{cases}
-\exp\biggl[-\dfrac12\ln p(t,\varepsilon)\,\biggr] \sin\biggl(\dfrac12\arctan\dfrac{\varepsilon}{t^2+\varepsilon^2}\biggr), &t^2+\varepsilon^2-t>0;\\
-\exp\biggl[-\dfrac12\ln p(t,\varepsilon)\,\biggr] \sin\biggl(\dfrac12\arctan\dfrac{\varepsilon}{t^2+\varepsilon^2}-\dfrac\pi2\biggr), &t^2+\varepsilon^2-t<0;\\
\exp\biggl(-\dfrac12\ln\dfrac{\varepsilon}{t^2+\varepsilon^2}\,\biggr)\sin\dfrac\pi4, &t^2+\varepsilon^2-t=0.
\end{cases}
\end{multline*}
Consequently,
\begin{equation}\label{im-limit-frac}
\lim_{\varepsilon\to0^+}\Im\biggl[\frac1{h(-t+i\varepsilon)}\biggr]=
\begin{cases}
\sqrt{\dfrac{t}{1-t}}, & 0<t<1;\\
\infty, & t=1;\\
0, & t>1.
\end{cases}
\end{equation}
\par
Let $D$ be a bounded domain with piecewise smooth boundary. The famous Cauchy integral formula (See~\cite[p.~113]{Gamelin-book-2001}) reads that if $f(z)$ is analytic on $D$, and $f(z)$ extends smoothly to the boundary of $D$, then
\begin{equation}\label{cauchy-formula=eq}
f(z)=\frac1{2\pi i}\oint_{\partial D}\frac{f(w)}{w-z}\td w,\quad z\in D.
\end{equation}
\par
For any fixed point $z\in\mathbb{C}\setminus(-\infty,0]$, choose $0<\varepsilon<1$ and $r>0$ such that $0<\varepsilon<|z|<r$, and consider the positively oriented contour $C(\varepsilon,r)$ in $\mathbb{C}\setminus(-\infty,0]$ consisting of the half circle $z=\varepsilon e^{i\theta}$ for $\theta\in\bigl[-\frac\pi2,\frac\pi2\bigr]$ and the half lines $z=x\pm i\varepsilon$ for $x\le0$ until they cut the circle $|z|=r$, which close the contour at the points $-r(\varepsilon)\pm i\varepsilon$, where $0<r(\varepsilon)\to r$ as $\varepsilon\to0$. See Figure~\ref{note-on-li-chen-conj.eps}.
\begin{figure}[htbp]
\includegraphics[width=0.75\textwidth]{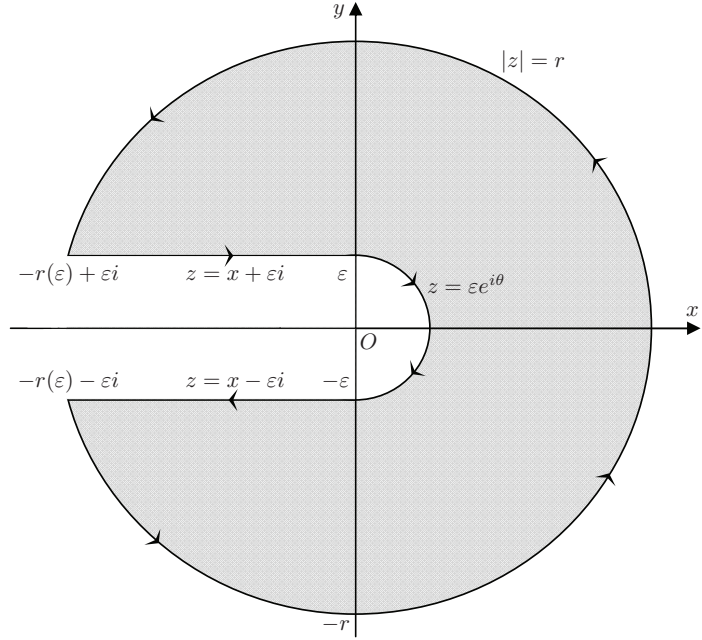}\\
\caption{The contour $C(\varepsilon,r)$}\label{note-on-li-chen-conj.eps}
\end{figure}
\par
By the above mentioned Cauchy integral formula, we have
\begin{align}
h(z)&=\frac1{2\pi i}\oint_{C(\varepsilon,r)}\frac{h(w)}{w-z}\td w\notag\\
&=\frac1{2\pi i}\biggl[\int_{\pi/2}^{-\pi/2}\frac{i\varepsilon e^{i\theta}h\bigl(\varepsilon e^{i\theta}\bigr)}{\varepsilon e^{i\theta}-z}\td\theta +\int_{-r(\varepsilon)}^0 \frac{h(x+i\varepsilon)}{x+i\varepsilon-z}\td x \label{h(z)-Cauchy-Apply}\\
&\quad+\int_0^{-r(\varepsilon)}\frac{h(x-i\varepsilon)}{x-i\varepsilon-z}\td x +\int_{\arg[-r(\varepsilon)-i\varepsilon]}^{\arg[-r(\varepsilon)+i\varepsilon]}\frac{ir e^{i\theta}h\bigl(re^{i\theta}\bigr)}{re^{i\theta}-z}\td\theta\biggr].\notag
\end{align}
By the limit~\eqref{zh(z)=0}, it follows that
\begin{equation}\label{zf(z)=0}
\lim_{\varepsilon\to0^+}\int_{\pi/2}^{-\pi/2}\frac{i\varepsilon e^{i\theta}h\bigl(\varepsilon e^{i\theta}\bigr)}{\varepsilon e^{i\theta}-z}\td\theta=0.
\end{equation}
In virtue of the limit~\eqref{sqrt-1=frac-z}, it can be derived that
\begin{equation}\label{big-circle-int=0}
\lim_{\substack{\varepsilon\to0^+\\r\to\infty}} \int_{\arg[-r(\varepsilon)-i\varepsilon]}^{\arg[-r(\varepsilon)+i\varepsilon]}\frac{ir e^{i\theta}h\bigl(re^{i\theta}\bigr)}{re^{i\theta}-z}\td\theta
=\lim_{r\to\infty}\int_{-\pi}^{\pi}\frac{ir e^{i\theta}h\bigl(re^{i\theta}\bigr)}{re^{i\theta}-z}\td\theta=2\pi i.
\end{equation}
Making use of the limits~\eqref{h(z)-h(z)-conjugate} and~\eqref{im-limit} yields that
\begin{multline}
\int_{-r(\varepsilon)}^0 \frac{h(x+i\varepsilon)}{x+i\varepsilon-z}\td x
+\int_0^{-r(\varepsilon)}\frac{h(x-i\varepsilon)}{x-i\varepsilon-z}\td x
=\int_{-r(\varepsilon)}^0 \biggl[\frac{h(x+i\varepsilon)}{x+i\varepsilon-z} -\frac{h(x-i\varepsilon)}{x-i\varepsilon-z}\biggr]\td x\\
\begin{aligned}
&=\int_{-r(\varepsilon)}^0\frac{(x-i\varepsilon-z)h(x+i\varepsilon) -(x+i\varepsilon-z)h(x-i\varepsilon)} {(x+i\varepsilon-z)(x-i\varepsilon-z)}\td x\\
&=\int_{-r(\varepsilon)}^0\frac{(x-z)[h(x+i\varepsilon)-h(x-i\varepsilon)] -i\varepsilon[h(x-i\varepsilon)+h(x+i\varepsilon)]} {(x+i\varepsilon-z)(x-i\varepsilon-z)}\td x\\
&=2i\int_{-r(\varepsilon)}^0\frac{(x-z)\Im h(x+i\varepsilon) -\varepsilon\Re h(x+i\varepsilon)} {(x+i\varepsilon-z)(x-i\varepsilon-z)}\td x\\
&\to2i\int_{-r}^0\frac{\lim_{\varepsilon\to0^+}\Im h(x+i\varepsilon)}{x-z}\td x\\
&=-2i\int^r_0\frac{\lim_{\varepsilon\to0^+}\Im h(-t+i\varepsilon)}{t+z}\td t\\
&\to-2i\int^\infty_0\frac{\lim_{\varepsilon\to0^+}\Im h(-t+i\varepsilon)}{t+z}\td t\\
&=2i\int_0^1\sqrt{\frac1t-1}\,\frac{\td t}{t+z} \label{level0lines}
\end{aligned}
\end{multline}
as $\varepsilon\to0^+$ and $r\to\infty$. Substituting equations~\eqref{zf(z)=0}, \eqref{big-circle-int=0}, and~\eqref{level0lines} into~\eqref{h(z)-Cauchy-Apply} and simplifying produce the integral representation~\eqref{thm-zhang-li-qi-eq}.
\par
Similarly, by the above mentioned Cauchy integral formula, we have
\begin{align}
\frac1{h(z)}&=\frac1{2\pi i}\oint_{C(\varepsilon,r)}\frac{1/h(w)}{w-z}\td w\notag\\
&=\frac1{2\pi i}\biggl[\int_{\pi/2}^{-\pi/2}\frac{i\varepsilon e^{i\theta}\bigr[1/h\bigl(\varepsilon e^{i\theta}\bigr)\bigr]}{\varepsilon e^{i\theta}-z}\td\theta +\int_{-r(\varepsilon)}^0 \frac{1/h(x+i\varepsilon)}{x+i\varepsilon-z}\td x \label{h(z)-frac-Cauchy-Apply}\\
&\quad+\int_0^{-r(\varepsilon)}\frac{1/h(x-i\varepsilon)}{x-i\varepsilon-z}\td x +\int_{\arg[-r(\varepsilon)-i\varepsilon]}^{\arg[-r(\varepsilon)+i\varepsilon]}\frac{ir e^{i\theta}\bigl[1/h\bigl(re^{i\theta}\bigr)\bigr]}{re^{i\theta}-z}\td\theta\biggr].\notag
\end{align}
From the limit~\eqref{frac-z=h(z)=0}, it follows that
\begin{equation}\label{zf(z)=0frac}
\lim_{\varepsilon\to0^+}\int_{\pi/2}^{-\pi/2}\frac{i\varepsilon e^{i\theta}\bigr[1/h\bigl(\varepsilon e^{i\theta}\bigr)\bigr]}{\varepsilon e^{i\theta}-z}\td\theta=0.
\end{equation}
By virtue of the limit~\eqref{sqrt-1=frac-z-recip}, it may be deduced that
\begin{equation}\label{big-circle-int=0frac}
\lim_{\substack{\varepsilon\to0^+\\r\to\infty}} \int_{\arg[-r(\varepsilon)-i\varepsilon]}^{\arg[-r(\varepsilon)+i\varepsilon]}\frac{ir e^{i\theta}\bigl[1/h\bigl(re^{i\theta}\bigr)\bigr]}{re^{i\theta}-z}\td\theta=2\pi i.
\end{equation}
Employing the limits~\eqref{h(z)-frac-conjugate} and~\eqref{im-limit-frac} yields that
\begin{align}
&\quad\int_{-r(\varepsilon)}^0 \frac{1/h(x+i\varepsilon)}{x+i\varepsilon-z}\td x
+\int_0^{-r(\varepsilon)}\frac{1/h(x-i\varepsilon)}{x-i\varepsilon-z}\td x \notag\\
&=2i\int_{-r(\varepsilon)}^0\frac{(x-z)\Im [1/h(x+i\varepsilon)] -\varepsilon\Re [1/h(x+i\varepsilon)]} {(x+i\varepsilon-z)(x-i\varepsilon-z)}\td x\notag\\
&\to2i\int_{-r}^0\frac{\lim_{\varepsilon\to0^+}\Im [1/h(x+i\varepsilon)]}{x-z}\td x \quad \text{as $\varepsilon\to0^+$}\notag\\
&\to-2i\int^\infty_0\frac{\lim_{\varepsilon\to0^+}\Im h(-t+i\varepsilon)}{t+z}\td t \quad \text{as $r\to\infty$}\notag\\
&=-2i\int_0^1\sqrt{\frac{t}{1-t}}\,\frac{\td t}{t+z}. \label{level0lines=frac}
\end{align}
Substituting equations~\eqref{zf(z)=0frac}, \eqref{big-circle-int=0frac}, and~\eqref{level0lines=frac} into~\eqref{h(z)-frac-Cauchy-Apply} and simplifying produce the integral representation~\eqref{thm-zhang-li-qi-eq-frac}.
\par
Adding~\eqref{thm-zhang-li-qi-eq} and~\eqref{thm-zhang-li-qi-eq-frac} leads to
\begin{align*}
H(z)&=2+\frac1\pi\int_0^1q(u)\frac{\td u}{u+z}\\
&=2+\frac1\pi\int_0^1q(u)\int_0^\infty e^{-(u+z)s}\td s\td u \\
&=2+\frac1\pi\int_0^\infty\biggl[\int_0^1 q(u)e^{-us}\td u\biggr]e^{-zs}\td s.
\end{align*}
Utilizing $q(u)=-q(1-u)$ for $u\in(0,1)$ or $q\bigl(\frac12+u\bigr)=-q\bigl(\frac12-u\bigr)$ for $u\in\bigl(0,\frac12\bigr)$ results in
\begin{multline*}
\int_0^1 q(u)e^{-us}\td u=\int_0^{1/2} q(u)e^{-us}\td u+\int_{1/2}^1 q(u)e^{-us}\td u\\
=\int_0^{1/2} q(u)e^{-us}\td u+\int_0^{1/2} q(1-u)e^{-(1-u)s}\td u\\
=\int_0^{1/2} q(u)\bigl[e^{-us}-e^{-(1-u)s}\bigr]\td u
=\int_0^{1/2} q(u)\bigl[1-e^{-(1-2u)s}\bigr]e^{-us}\td u
\ge0
\end{multline*}
or
\begin{align*}
\int_0^1 q(u)e^{-us}\td u&=\int_0^{1/2} q\biggl(\frac12-u\biggr)e^{-(1/2-u)s}\td u +\int_0^{1/2}q\biggl(\frac12+u\biggr)e^{-(1/2+u)s}\td u\\
&=\int_0^{1/2} q\biggl(\frac12-u\biggr)\bigl[e^{-(1/2-u)s}-e^{-(1/2+u)s}\bigr]\td u\\
&=\int_0^{1/2} q\biggl(\frac12-u\biggr)\bigl(e^{us}-e^{-us}\bigr)e^{-s/2}\td u\\
&\ge0.
\end{align*}
The proof of Lemma~\ref{lem2-zhang-li-qi-eq} is thus completed.
\end{proof}

\begin{proof}[Proof by Stieltjes-Perron inversion formula]
The property~(x) in~\cite[Theorem~1.3]{Kalugin-Jeffrey-Corless-Borwein} formulates that if $f\in\mathcal{S}$ then $f^\alpha(0^+)-f^\alpha\bigl(\frac1t\bigr)\in\mathcal{S}$ for $0\le\alpha\le1$.
Since $h(t)\in\mathcal{S}$, see~\eqref{h(t)inS-eq}, and, by the property~(i) in~\cite[Theorem~1.3]{Kalugin-Jeffrey-Corless-Borwein}, $\frac1{h(1/t)}\in\mathcal{S}$, replacing $f$ by $\frac1{h(1/t)}$, making use of the easy fact that $f(0^+)=\lim_{t\to0^+}f(t)=1$, and letting $\alpha=1$ yield $1-\frac1{h(t)}\in\mathcal{S}$.
\par
For a Stieltjes function $f$ given by~\eqref{dfn-stieltjes}, by the Stieltjes-Perron inversion formula in~\cite[p.~591]{Henrici-book-77}, we can determine the scalars $a=\lim_{x\to0^+}[xf(x)]$ and $b=\lim_{x\to\infty}f(x)$ and the measure
\begin{equation}
\mu(s)=-\frac1\pi\lim_{t\to0^+}\Im\int_{-\infty}^{-s}f(u+ti)\td u,
\end{equation}
as done in~\cite{CBerg, Kalugin-Jeffrey-Corless-640327}. Specially, for the function $h(x)$, since $a=\lim_{x\to0^+}[xh(x)]=0$ and $b=\lim_{x\to\infty}h(x)=1$, we have
\begin{equation}\label{h(z)=Phi-eq}
h(z)=1+\int_0^\infty\frac{\td\Phi(u)}{u+z}
\end{equation}
for $|\arg z|<\pi$, where
\begin{equation*}
\Phi(u)=\frac1\pi\lim_{s\to0^+}\int_u^\infty\Im\sqrt{1-\frac1{\tau-is}}\,\td\tau =-\frac1\pi\int_u^\infty\sqrt{\frac1\tau-1}\,\td\tau
\end{equation*}
when $0<\tau<1$ and $\Phi(u)=0$ when $\tau>1$ because taking $s\to0^+$ we obtain
\begin{equation*}
\Im\sqrt{1-\frac1{\tau-is}}\,=\Im\sqrt{\frac{-[\tau(1-\tau)-s^2]-si}{\tau^2+s^2}}\, \to-\sqrt{\frac{\tau(1-\tau)}{\tau^2}}\,
\end{equation*}
when $0<\tau<1$ and $\Im\sqrt{1-\frac1{\tau-si}}\,\to0$ when $\tau>1$. Thus we find
\begin{equation*}
\Phi'(u)=\frac1\pi\sqrt{\frac1u-1}\,
\end{equation*}
when $0<u<1$ and $\Phi'(u)=0$ when $u>1$. Substituting $\Phi'(u)$ in the representation~\eqref{h(z)=Phi-eq} results in the formula~\eqref{thm-zhang-li-qi-eq}.
\par
The formula~\eqref{thm-zhang-li-qi-eq-frac} for $\frac1{h(z)}$ or for $1-\frac1{h(z)}$ can be derived in a similar way as above.
\par
The rest is the same as in the first proof. Lemma~\ref{lem2-zhang-li-qi-eq} is proved once again.
\end{proof}

\section{The harmonic mean is a Bernstein function}

Our results on the harmonic mean $H_{x,y}(t)$ may be stated as the theorem below.

\begin{thm}\label{harmonic-bernstein-thm}
The harmonic mean $H_{x,y}(t)$ defined by~\eqref{harmnic=x-y+t-eq} is a Bernstein function of $t$ on $(-\min\{x,y\},\infty)$ for $x,y>0$ with $x\ne y$ and has the integral representation
\begin{equation}\label{H-int-reprent}
H_{x,y}(t)=H(x,y)+t+\frac{(x-y)^2}4 \int_0^\infty\bigl(1-e^{-tu}\bigr) e^{-(x+y)u/2}\td u.
\end{equation}
Consequently,
\begin{align}\label{H-int-reprent-A=H}
H(x,y)&=A(x,y)-\frac{(x-y)^2}2 \int_0^\infty e^{-(x+y)u}\td u\\
H(s,y+s)&=s+\frac{y^2}4 \int_0^\infty\bigl(1-e^{-su}\bigr) e^{-yu/2}\td u,\quad s>0.\label{H(s=y+s)-int-eq}
\end{align}
\end{thm}

\begin{proof}
The harmonic mean $H_{x,y}(t)$ meets
\begin{equation}\label{H-derivat-first}
H_{x,y}'(t)=\frac{2\bigl[x^2+y^2+2(x+y)t+2t^2\bigr]}{(x+y+2t)^2} =1+\frac{(x-y)^2}{(x+y+2t)^2}>1.
\end{equation}
It is obvious that the derivative $H_{x,y}'(t)$ is completely monotonic with respect to $t$. As a result, the harmonic mean $H_{x,y}(t)$ is a Bernstein function of $t$ on $(-\min\{x,y\},\infty)$ for $x,y>0$ with $x\ne y$.
\par
In~\cite[p.~255, 6.1.1]{abram}, it was listed that, for $\Re z>0$ and $\Re k>0$, the classical Euler gamma function
\begin{equation}\label{Gamma(z)=k-z-int}
\Gamma(z)=k^z\int_0^\infty t^{z-1}e^{-kt}\td t.
\end{equation}
This formula can be rearranged as
\begin{equation}\label{Gamma(z)=k-z-int-rearr}
\frac1{z^w}=\frac1{\Gamma(w)}\int_0^\infty t^{w-1}e^{-zt}\td t
\end{equation}
for $\Re z>0$ and $\Re w>0$.
Combining~\eqref{Gamma(z)=k-z-int-rearr} with~\eqref{H-derivat-first} yields
\begin{equation}\label{H-derivat-first-int}
H_{x,y}'(t)=1+(x-y)^2\int_0^\infty ue^{-(x+y+2t)u}\td u,
\end{equation}
and so, by integrating with respect to $t\in(0,s)$ on both sides of~\eqref{H-derivat-first-int}, the formula~\eqref{H-int-reprent} follows.
\par
Letting $s\to\infty$ on both sides of~\eqref{H-int-reprent} and using the limit
$
\lim_{s\to\infty}[H_{x,y}(s)-s]=A(x,y)
$
generate the formula~\eqref{H-int-reprent-A=H}.
\par
Taking $x\to0^+$ in~\eqref{H-int-reprent} produces~\eqref{H(s=y+s)-int-eq}. Theorem~\ref{harmonic-bernstein-thm} is thus proved.
\end{proof}

\begin{rem}
By~\cite[pp.~161\nobreakdash--162, Theorem~3]{Chen-Qi-Srivastava-09.tex} or~\cite[p.~45, Proposition~5.17]{Schilling-Song-Vondracek-2010}, it can be derived that the reciprocal of the harmonic mean $H_{x,y}(t)$, that is, the function $\frac1{A(1/(x+t),1/(y+t))}$, is logarithmically completely monotonic.
\par
This logarithmically complete monotonicity can also be proved by considering
\begin{equation*}
[\ln H_{x,y}(t)]'=\frac{x^2+y^2+2(x+y)t+2t^2}{(x+t)(y+t)(x+y+2t)} =\frac12\biggl(\frac1{x+t}+\frac1{y+t}\biggr) \biggl[1+\frac{(x-y)^2}{(x+y+2t)^2}\biggr]
\end{equation*}
and that the product and sum of finitely many completely monotonic functions are also completely monotonic functions.
\par
Moreover, from~\eqref{H-derivat-first}, it follows readily that $H_{x,y}(t)-t$ is an increasing function in $t\in(-\min\{x,y\},\infty)$ for $x,y>0$ with $x\ne y$.
\end{rem}

\section{The geometric mean is a Bernstein function}

Our results on the geometric mean $G_{x,y}(t)$ can be summarized as two theorems.

\begin{thm}\label{geometric-bernstein-thm}
Let $x,y>0$ with $x\ne y$. Then the geometric mean $G_{x,y}(t)$ defined by~\eqref{G(x=y)(t)} is a Bernstein function of $t$ on $(-\min\{x,y\},\infty)$.
\end{thm}

We supply three proofs of Theorems~\ref{geometric-bernstein-thm}.

\begin{proof}[First proof]
By a direct differentiation, we have
\begin{equation*}
G_{x,y}'(t)=\sqrt{\frac{x+t}{y+t}}\,\frac{x+y+2t}{2(x+t)}.
\end{equation*}
Taking the logarithm on both sides of the above equality creates
\begin{equation}\label{lnG(x-y)(t)}
\ln G_{x,y}'(t)=\frac12\ln\frac{x+t}{y+t}+\ln\frac{x+y+2t}{2(x+t)}.
\end{equation}
In~\cite[p.~230, 5.1.32]{abram}, it was collected that for $a>0$ and $b>0$,
\begin{equation}\label{ln-frac}
\ln\frac{b}a=\int_0^\infty\frac{e^{-au}-e^{-bu}}u\td u.
\end{equation}
Using this formula in~\eqref{lnG(x-y)(t)} leads to
\begin{equation*}
\ln G_{x,y}'(t)=\int_0^\infty\frac{e^{-(x+t)v}+e^{-(y+t)v}-2e^{-v[(x+t)+(y+t)]/2}}{2v}\td v.
\end{equation*}
Since the function $e^{-t}$ is convex on $\mathbb{R}$, we have
\begin{equation*}
e^{-(x+t)v}+e^{-(y+t)v}-2e^{-v[(x+t)+(y+t)]/2}\ge0.
\end{equation*}
Therefore, we have
\begin{equation*}
[\ln G_{x,y}'(t)]^{(k)}=\frac{(-1)^k}2 \int_0^\infty\bigr\{e^{-(x+t)v}+e^{-(y+t)v}-2e^{-v[(x+t)+(y+t)]/2}\bigr\} v^{k-1}\td v.
\end{equation*}
This means that the derivative $G_{x,y}'(t)$ is logarithmically completely monotonic, and so it is also completely monotonic. As a result, the geometric mean $G_{x,y}(t)$ is a Bernstein function.
\end{proof}

\begin{proof}[Second proof]
It is clear that the geometric mean $G_{x,y}(t)$ satisfies
\begin{equation}\label{G(x=y)(t)-deriv}
G_{x,y}'(t)=\frac12\biggl(\sqrt{\frac{x+t}{y+t}}\, +\sqrt{\frac{y+t}{x+t}}\,\biggr) =\frac12\biggl(\sqrt{u}+\frac1{\sqrt{u}}\biggr) \triangleq f(u)
\end{equation}
and
\begin{equation}\label{G(x=y)(t)-deriv-log}
[\ln G_{x,y}(t)]'=\frac12\biggl(\frac1{x+t}+\frac1{y+t}\biggr),
\end{equation}
where
\begin{equation}\label{var-u-def-eq}
u\triangleq u_{x,y}(t)=\frac{x+t}{y+t}=1+\frac{x-y}{y+t}.
\end{equation}
If $0<x<y$, then $0<u_{x,y}(t)<1$ for $t\in(-x,\infty)$ and $u_{x,y}'(t)=\frac{y-x}{(y+t)^2}$ is completely monotonic in $t\in(-x,\infty)$. On the other hand, the function $f(u)$ is positive and
\begin{align*}
f^{(i)}(u)&=\frac12\biggl[(-1)^{i-1}\frac{(2i-3)!!}{2^i}u^{-(2i-1)/2} +(-1)^{i}\frac{(2i-1)!!}{2^{i}}u^{-(2i+1)/2}\biggr]\\
&=\frac{(-1)^{i}(2i-3)!!}{2^{i+1}} \frac1{u^{(2i-1)/2}} \biggl(\frac{2i-1}{u}-1\biggr)
\end{align*}
for $i\in\mathbb{N}$, which implies that the function $f(u)$ is completely monotonic on $(0,1)$; A ready modification of a conclusion in~\cite[p.~83]{bochner} yields the following conclusion: If $g$ and $h'$ are completely monotonic functions such that $g(h(x))$ is defined on an interval $I$, then $x\mapsto g(h(x))$ is also completely monotonic on $I$; So, when $y>x>0$, the derivative $G_{x,y}'(t)$ is completely monotonic and the geometric mean $G_{x,y}(t)$ is a Bernstein function. Consequently, considering the symmetric property $G_{x,y}(t)=G_{y,x}(t)$, it is easily obtained that the geometric mean $G_{x,y}(t)$ for $t\in(-\min\{x,y\},\infty)$ with $x\ne y$ is a Bernstein function.
\end{proof}

\begin{rem}
From the equality in~\eqref{G(x=y)(t)-deriv}, it is easy to derive that the function $G_{x,y}(t)-t$ is increasing in $t\in(-\min\{x,y\},\infty)$ for $x,y>0$ with $x\ne y$.
\par
From~\eqref{G(x=y)(t)-deriv-log}, it is immediate to deduce that the reciprocal of the geometric mean $G_{x,y}(t)$ is a logarithmically completely monotonic function of $t\in(-\min\{x,y\},\infty)$ for $x,y>0$ with $x\ne y$.
\end{rem}

\begin{proof}[Third proof]
By~\eqref{G(x=y)(t)-deriv} and~\eqref{var-u-def-eq}, it follows that
\begin{equation}\label{G-deriv=H-relation}
G_{x,y}'(t)=\frac12\biggl[h\biggl(\frac{y+t}{x-y}\biggr)+\frac1{h\bigl(\frac{y+t}{x-y}\bigr)}\biggr] =\frac12H\biggl(\frac{y+t}{x-y}\biggr)
\end{equation}
and
\begin{equation*}
[G_{x,y}'(t)]^{(i)}=\frac1{2(x-y)^i}H^{(i)}\biggl(\frac{y+t}{x-y}\biggr)
\end{equation*}
for $i\in\{0\}\cup\mathbb{N}$.
By the formula~\eqref{H(t)-expression} in Lemma~\ref{Geom-Mean-CM-lem-h(t)}, we have
\begin{multline*}
[G_{x,y}'(t)]^{(i)}=\frac{(-1)^i}{2^{i+1}\bigl(\frac{y+t}{x-y}\bigr)^{i+1} (x+t)^{i}h\bigl(\frac{y+t}{x-y}\bigr)}\\
\times\sum_{k=0}^{i-1}\frac{(i-1)!(i+1)!(2i-2k-1)!!} {(i-k-1)!(i-k+1)!k!}2^{k}\biggl(\frac{y+t}{x-y}\biggr)^k,
\end{multline*}
which means that, when $x>y$, the derivative $G_{x,y}'(t)$ is completely monotonic. Since $G_{x,y}(t)=G_{y,x}(t)$, when $x<y$, the derivative $G_{y,x}'(t)$ is also completely monotonic. This implies that the geometric mean $G_{x,y}(t)$ is a Bernstein function of $t\in(-\min\{x,y\},\infty)$.
\end{proof}

\begin{thm}\label{geometric-integral-thm}
For $x>y>0$ and $z\in\mathbb{C}\setminus(-\infty,-y]$, the geometric mean $G_{x,y}(z)$ has the integral representation
\begin{equation}\label{G(x-y)-integ-exp}
G_{x,y}(z)=G(x,y)+z+\frac{x-y}{2\pi} \int_0^\infty\frac{\rho((x-y)s)}s e^{-ys}\bigl(1-e^{-sz}\bigr)\td s,
\end{equation}
where the function $\rho$ is defined by~\eqref{rho-funct-dfn}. Consequently, the geometric mean $G_{x,y}(t)$ is a Bernstein function of $t$ on $(-\min\{x,y\},\infty)$.
\end{thm}

\begin{proof}
For $x>y>0$ and $z\in\mathbb{C}\setminus(-\infty,-y]$, making use of
\begin{equation*}
G_{x,y}'(z)=\frac12\biggl[h\biggl(\frac{y+z}{x-y}\biggr)+\frac1{h\bigl(\frac{y+z}{x-y}\bigr)}\biggr] =\frac12H\biggl(\frac{y+z}{x-y}\biggr)
\end{equation*}
and~\eqref{H(z)=int-expression} gives
\begin{equation*}
G_{x,y}'(z)=1+\frac1{2\pi}\int_0^\infty\rho(s)\exp\biggl(-\frac{y+z}{x-y}s\biggr)\td s.
\end{equation*}
Integrating with respect to $z$ from $0$ to $w$ on both sides of the above equation and interchanging the order of integrals yield
\begin{align*}
G_{x,y}(w)-G_{x,y}(0)&=w+\frac{x-y}{2\pi}\int_0^\infty\frac{\rho(s)}s\exp\biggl(-\frac{ys}{x-y}\biggr) \biggl[1-\exp\biggl(-\frac{sw}{x-y}\biggr)\biggr]\td s\\
&=w+\frac{x-y}{2\pi}\int_0^\infty\frac{\rho((x-y)s)}s e^{-ys}\bigl(1-e^{-ws}\bigr)\td s.
\end{align*}
Since $G_{x,y}(0)=G(x,y)$, the integral representation~\eqref{G(x-y)-integ-exp} is readily deduced.
\par
By the characterization expressed by~\eqref{Bernstein-Characterize} and the integral representation~\eqref{G(x-y)-integ-exp} applied to $z=t\in(-\min\{x,y\},\infty)$, it is immediate to see that the geometric mean $G_{x,y}(t)$ is a Bernstein function of $t$ on $(-\min\{x,y\},\infty)$.
\end{proof}

\begin{rem}
Taking $z\to\infty$ in~\eqref{G(x-y)-integ-exp} and using
$
\lim_{z\to\infty}[G_{x,y}(z)-z]=A(x,y)
$
yield
\begin{equation}\label{G(x-y)-integ-exp-A}
A(x,y)=G(x,y)+\frac{x-y}{2\pi} \int_0^\infty\frac{\rho((x-y)s)}s e^{-ys}\td s\ge G(x,y).
\end{equation}
The equality in~\eqref{G(x-y)-integ-exp-A} is valid if and only if $x=y$. This gives a new proof of the fundamental and well known AG mean inequality.
\end{rem}

\end{document}